\documentclass[12pt, reqno,sumlimits]{amsart}
\usepackage{amssymb,amscd, epsfig, mathrsfs, xypic, amsmath,amsthm, color}%dsfont, txfonts,  
\usepackage{hyperref}
\usepackage{blindtext}
\usepackage{enumerate}

\xyoption{all}
\newtheorem{thm}{Theorem}[section]  
 
\newtheorem{lem}[thm]{Lemma}
\newtheorem{prop}[thm]{Proposition} 
\newtheorem{df-pr}[thm]{Definition-Proposition}

\theoremstyle{definition}
\newtheorem{defn}[thm]{Definition}
 
\newtheorem{rem}[thm]{Remark}

\newtheorem{exm}[thm]{Example}

\numberwithin{equation}{section}

\newcommand{\Ess}{{\mathcal{E}ss}}

\newcommand{\ee}{{e}}

\newcommand{\LL}{{\mathbb L}}

\newcommand{\PP}{{\mathbb P}}

\newcommand{\ZZ}{{\mathbb Z}}

\newcommand{\sfn }{{\mathsf n}}

\newcommand{\bfk}{{\mathbf k }}

\newcommand{\bfm}{{\mathbf m }}

\newcommand{\bff}{{\mathbf f }}

\newcommand{\bfp}{{\mathbf p}}

\newcommand{\bfs}{{\mathbf s}}

\newcommand{\CH}{{C\!H}}
\newcommand{\CK}{{C\!K}}

\newcommand{\calL}{{\mathcal L}}

\newcommand{\calQ}{{\mathcal Q}}

\newcommand{\scA}{{\mathscr A}}

\newcommand{\scO}{{\mathscr O}}
\newcommand{\scP}{{\mathscr P}}

\newcommand{\scS}{{\mathscr S}}

\newcommand{\surj}{\twoheadrightarrow}

\newcommand{\Fl}{{Fl}}

\newcommand{\gr}{\operatorname{gr}}

\newcommand{\rk}{{\operatorname{rk}}}

\newcommand{\supp}{{\operatorname{supp}}}

\newcommand{\FSVT}{{\it FSVT}}

\begin{document} 
\title{Vexillary degeneracy loci classes in K-theory and algebraic cobordism}
\author{Thomas Hudson and Tomoo Matsumura}
\date{} 
\maketitle 
\begin{abstract}
In this paper, we prove determinant formulas for the $K$-theory classes of the structure sheaves of degeneracy loci classes associated to vexillary permutations in type $A$. As a consequence we obtain determinant formulas for Lascoux--Sch{\"u}tzenberger's double Grothendieck polynomials associated to vexillary permutations. Furthermore, we generalize the determinant formula to algebraic cobordism. 
\end{abstract}
%\tableofcontents
%%%%%%%%%%%%%%%%%%%%%%%%%%%%%%%
%%%%%%%%%%%%%%%%%%%%%%%%%%%%%%%
\section{Introduction}
%%%%%%%%%%%%%%%%%%%%%%%%%%%%%%%
%%%%%%%%%%%%%%%%%%%%%%%%%%%%%%%
Let $E_1 \subset \dots \subset E_{n-1} \to H_{n-1} \surj \cdots \surj H_2 \surj H_1$ be a sequence of maps of vector bundles over a nonsingular variety $X$ where the subscripts indicate the rank. For a permutation $w\in S_{n}$, we define the degeneracy locus $X_w$ in $X$ by
\[
X_w:=X_w(E_{\bullet} \to H_{\bullet}) := \{ x \in X \ |\ \rk(F_q|_x \to H_p|_x) \leq r_w(p,q) \ \forall p,q\}
\]
where $r_w(p,q)=\sharp\{i \leq p \ |\ w(i)\leq q\}$. 
A permutation $w \in S_n$ is called \emph{vexillary} if it avoids the pattern $(2143)$, \textit{i.e.} there is no $a<b<c<d$ such that $w(b)<w(a)<w(d)<w(c)$. The primary goal of this article is to prove determinantal formulas for the $K$-theory class $[\scO_{X_w}]$ of the structure sheaf of $X_w$, where $w$ is a vexillary permutation (Theorem \ref{thmMainA}). Such formulas, which is written in terms of Segre classes, are later generalized to algebraic cobordism (Theorem \ref{thmMainAALG}). Our proof is based on our joint work \cite{HIMN} with Ikeda and Naruse and on \cite{HudsonMatsumura}, where determinantal formulas were obtained for Grassmannian permutations. Independently from this work, Anderson \cite{Anderson2016} also proved a determinant formula for $[\scO_{X_w}]$ in terms of Chern classes with a similar method.

Lascoux and Sch{\"u}tzenberger (\cite{Lascoux1}, \cite{LascouxSchutzenberger3}) introduced the double Grothendieck polynomial $G_w(x|b)$ to represent the $K$-theory classes of the structure sheaves of Schubert varieties and Fomin--Kirillov  (\cite{GrothendieckFomin}, \cite{DoubleGrothendieckFomin}) gave their combinatorial description in terms of pipe dreams or rc graphs. In the case when $w$ is vexillary, Knutson--Miller--Yong \cite{KnutsonMillerYong} described $G_w(x|b)$ in terms of \emph{flagged set-valued tableaux}, unifying the work of Wachs \cite{Wachs} on flagged tableaux, and Buch \cite{BuchLRrule} on set-valued tableaux. If the codimension of $X_w$ coincides with the length of $w$, then we know from the works of Fulton and Lascoux \cite{FultonLascoux} and Buch \cite{BuchQuiver} that $[\scO_{X_w}] = G_w(x|b)$ where we set $x_i = c_1(\ker(H_i\to H_{i-1}))$ and $b_i=c_1((E_i/E_{i-1})^{\vee})$. As a consequence, our formula also proves determinantal formulas of $G_w(x|b)$, generalizing the Jacobi-Trudi type formulas of double Grothendieck polynomials associated to Grassmann permutations obtained in \cite{HIMN} and \cite{HudsonMatsumura}. Note that in \cite{Lenart2000}, \cite{KirillovDunkl}, \cite{Kirillov2015}, and \cite{Yeliussizov}, one can find different determinantal formulas of Grothendieck polynomials for Grassmannian permutations as well as to certain vexillary permutations. Those formulas are in terms of complete/elementary symmetric polynomials while ours are in terms of Grothendieck polynomials associated to one row partitions.

Let us now recall the corresponding results for cohomology. Lascoux and Sch{\"u}tzenberger defined double Schubert polynomials indexed by permutations in \cite{ClassesLascoux} and \cite{SchubertLascoux}. Fulton \cite{FlagsFulton} showed that they give the cohomology classes of the degeneracy loci
given by rank conditions associated to permutations. Lascoux and Sch{\"u}tzenberger also introduced vexillary permutations where the associated Schubert polynomials are given by generalized Schur determinants. In \cite{FlagsFulton}, Fulton gave a simple description of the degeneracy loci given by vexillary permutations (see also \cite{AndersonFulton} and \cite{AndersonFulton2}) and proved the determinant formula for the cohomology classes of such degeneracy loci. For Grassmannian permutations (which are examples of vexillary permutations), the corresponding (double) Schubert polynomials coincide with the (factorial) Schur polynomials, and their determinant formula is known as the Jacobi--Trudi formula or the determinant formulas of Kempf--Laksov \cite{KempfLaksov} and Damon \cite{Damon1974}, which further specializes to Giambelli--Thom--Porteous formula.

The paper is organized as follows. In Section \ref{SecCK}, we review some of the basics on Segre classes in connective $K$-theory. In Section \ref{SecVex}, we recall the partition associated to a vexillary permutation $w$ and a flagging set of $w$, which corresponds to an (inflated) triple of type $A$ in \cite{AndersonFulton2}. In Section \ref{SecMain}, we describe our main theorem in terms of the partition and a flagging set. The proof is based on the construction of a resolution using the data of the flagging set and the computation of the class of $X_w$ as the pushforward of the class of the resolution following \cite{HIMN}. In Section \ref{SecGro}, we describe the corresponding determinant formula for the vexillary double Grothendieck polynomials. In Section \ref{SecAlgCob}, we also take into consideration the algebraic cobordism of Levine--Morel \cite{LevineMorel} and describe the corresponding classes as an infinite linear combination of Schur determinants whose coefficients are given by a certain power series defined by means of the universal formal group law.

In this paper, all schemes and varieties are assumed to be quasi-projective over an algebraically closed field $\bfk$ of characteristic zero, unless otherwise stated.

%%%%%%%%%%%%%%%%%%%%%%%%%%%%%%%
%%%%%%%%%%%%%%%%%%%%%%%%%%%%%%%
\section{Connective $K$-theory and Segre classes}\label{SecCK}
%%%%%%%%%%%%%%%%%%%%%%%%%%%%%%%%%%
%%%%%%%%%%%%%%%%%%%%%%%%%%%%%%%%%%
Connective $K$-theory, denoted by $\CK^*$, is an example of oriented cohomology theory built out of the algebraic cobordism of Levine and Morel. For the detailed construction we refer the reader to \cite{DaiLevine}, \cite{Hudson}, and \cite{LevineMorel}. Connective $K$-theory assigns to a smooth variety $X$  a commutative graded algebra $\CK^*(X)$  over the coefficient ring $\CK^*({pt})$. Here the ring $\CK^*({pt})$ is isomorphic to the polynomial ring $\ZZ[\beta]$ by setting $\beta$ to be $(-1)$ times the class of degree $-1$ obtained by pushing forward the fundamental class of the projective line along the structural morphism $\PP^1 \to {pt}$. The $\ZZ[\beta]$-algebra $\CK^*(X)$ specializes to the Chow ring $\CH^*(X)$ and the Grothendieck ring $K(X)$ by respectively setting $\beta$ equal to $0$ and $-1$. For any closed equidimensional subvariety $Y$ of $X$, there exists an associated fundamental class  $[Y]_{\CK^*}$ in $\CK^*(X)$ which specializes to the class $[Y]$ in $\CH^*(X)$ and also to the class of the structure sheaf $\mathcal{O}_Y$ of $Y$ in $K(X)$. In the rest of the paper, we denote the fundamental class of $Y$ in $\CK^*(X)$ by $[Y]$ instead of $[Y]_{\CK^*}$.

As a feature of any oriented cohomology theory, connective $K$-theory admits a theory of Chern classes. For a line bundle $L$, its first Chern class $c_1(L)$ corresponds to the usual Chern class in $\CH^*(X)$ and to the class $1-[L^{\vee}]$ in $K(X)$ under the specialization $\beta=0$ and $\beta=-1$ respectively. For line bundles $L_1$ and $L_2$ over $X$, their 1st Chern classes $c_1(L_i)\in  \CK^1(X)$ satisfy the following equality:
\[
c_1(L_1\otimes L_2)=c_1(L_1)+c_1(L_2)+\beta c_1(L_1)c_1(L_2).
\]
It follows that $c_1(L^{\vee}) = \frac{-c_1(L)}{1+\beta c_1(L)}$. For a variable $u$, we denote $\bar u := \frac{-u}{1+\beta u}$. The reader should be aware that our sign convention for $\beta$ is opposite to the one used in the references \cite{DaiLevine}, \cite{Hudson}, \cite{LevineMorel}. 
%%%%%%%%%%%%%%%%%%%%%%%%%%%%%%%%%%
For computations it is convenient  to combine the Chern classes into a Chern polynomial $c(E;u):=\sum_{i=0}^{\ee} c_i(E) u^i$. For each virtual bundle $E-F$ as an element of $K(X)$, one has the relative Chern classes $c_i(E-F)$ defined by $c(E-F;u):=c(E;u)/c(F;u)$.
%%%%%%%%%%%%%%%%%%%%%%%%%%%%%%%%%%
\begin{defn} \label{defnSeg}
We define the \emph{Segre classes} for a virtual bundle $E-F$ by the following:
\begin{equation}\label{segre vir}
\scS(E-F;u):=\sum_{m\in \ZZ} \scS_{m}(E-F) u^{m}= \frac{1}{1 + \beta u^{-1}} \frac{c(E - F;\beta)}{c(E-F;-u)}.
\end{equation}
\end{defn}
Consequently the Segre classes satisfy the following identities ({\it cf.} \cite{HIMN}, \cite{HudsonMatsumura}): 
%\begin{eqnarray}
%\scS(E-F;u) 
%&=&\scS(E;u)\frac{c(F;-u)}{c(F;\beta)}\\
%&=&\scS(E;u)c(F^{\vee};u+\beta).
%\end{eqnarray}
%Equivalently we have
\begin{eqnarray}
\scS_m(E-F) 
&=&\sum_{p=0}^{\rk F} (-1)^p c_p(F)\scS_{m-p}(E)\frac{1}{c(F;\beta)}  \label{eqrelS1}\\
&=&\sum_{p=0}^{\rk F} c_p(F^{\vee}) \sum_{q=0}^p \binom{p}{q}\beta^q \scS_{m-p+q}(E)\label{eqrelS2}.
\end{eqnarray}

%%%%%%%%%%%%%%%%%%%%%%%%%%%%%%%%%
The following proposition was obtained in \cite{HIMN} based on the results in  \cite{BuchQuiver} and \cite{Vishik}.
\begin{prop}\label{push of tensor} 
Let $\pi: \PP^*(E)\rightarrow X$ be the dual projective bundle of a rank $e$ vector bundle $E\to X$ and let $\tau$ be the first Chern class of its tautological quotient line bundle $\calQ$. Let $F$ be a vector bundle over $X$ of rank $f$ and denote its pullback to $\PP^*(E)$ also by $F$. We have 
\[
\pi_*\left(\tau^sc_f(\calQ \otimes F^{\vee})\right) = \scS_{s+f-e+1}(E-F).
\]
\end{prop}
%%%%%%%%%%%%%%%%%%%%%%%%%%%%%%%%%%
\begin{rem}\label{remLine}
When $E$ is a line bundle, we have $c_1(E)^sc_{f}(E \otimes F^{\vee}) = \scS_{f+s}(E-F)$ by regarding $E$ as the tautological quotient line bundle of $\PP^*(E)$.
\end{rem}
%%%%%%%%%%%%%%%%%%%%%%%%%%%%%%%
%%%%%%%%%%%%%%%%%%%%%%%%%%%%%%%
%%%%%%%%%%%%%%%%%%%%%%%%%%%%%%%
\section{Vexillary permutation}\label{SecVex}
%%%%%%%%%%%%%%%%%%%%%%%%%%%%%%%
%%%%%%%%%%%%%%%%%%%%%%%%%%%%%%%
%%%%%%%%%%%%%%%%%%%%%%%%%%%%%%%
Let $S_n$ be a permutation group of $\{1,\dots,n\}$ and $w \in S_n$.  The \emph{rank function} of $w$ is defined by
\[
r_w(p,q) := \sharp\{ i\leq p  |\ w(i) \leq q\}.
\]
The diagram of $w$, denoted by $D(w)$, is defined by
\[
D(w):=\{(p,q)\in \{1,\dots,n\}\times \{1,\dots,n\} \ |\ \pi(p) >q, \ \mbox{and} \ \pi^{-1}(q)>p\}.
\]
We call an element of the grid $\{1,\dots, n\} \times \{1,\dots,n\}$ a box. Two elements $(p,q)$ and $(p',q')$ in $D(w)$ are said to be connected 
we can find a sequence of adjacent boxes connecting them that are contained in $D(w)$. The \emph{essential set} $\Ess(w)$ of $w$ is the subset of $D(w)$ given by
\[
\Ess(w):=\{(p,q) \ |\ p,q\leq n-1, w(p)>q, w(p+1)\leq q, w^{-1}(q)>p, \mbox{ and } w^{-1}(q+1)\leq p\}.
\]
A permutation $w \in S_n$ is called \emph{vexillary} if it avoids the pattern $(2143)$, \textit{i.e.} there is no $a<b<c<d$ such that $w(b)<w(a)<w(d)<w(c)$. A characterization of vexillary permutations in terms of the essential sets has been given by Fulton \cite[Proposition 9.6]{FlagsFulton}. In particular, $w$ is vexillary if and only if the boxes in $\Ess(w)$ are placed along the direction going from north-east to south-west, {\it i.e.} for any $(p,q), (p',q') \in \Ess(w)$, $p\geq p'$ implies $q\leq q'$.
%%%%%%%%%%%%%%%%%%%%%%%%%%%%%%%
%\begin{prop}[Proposition 9.6 \cite{FlagsFulton}]\label{propFulton}
%Let $\tilde{p}_i, \tilde{q}_i$ and $r_i$ be non-negative integers for $i=1,\dots, s$ such that
%\begin{equation}
% \tilde{p}_1 \leq \cdots \leq \tilde{p}_k, \ \ \ \ \ \tilde{q}_1 \geq \cdots \geq \tilde{q}_k;
%\end{equation}
%\begin{equation}
%0<\tilde{p}_1-r_1<\cdots <\tilde{p}_k-r_k, \ \ \ \ \ \tilde{q}_1-r_1>\cdots >\tilde{q}_k-r_k>0, \ \ 
%\end{equation}
%For any $n \geq \tilde{p}_s+\tilde{q}_1$, there is a unique permutation $w \in S_n$ with 
%\[
%\Ess(w)=\{(\tilde{p}_i,\tilde{q}_i), i=1,\dots,k\}
%\]
%and
%\[
%r_i=r_w(\tilde{p}_i,\tilde{q}_i), \ \ \ i=1,\dots,k.
%\] 
%Furthermore, such permutation $w$ is always vexillary and every vexillary permutation arises this way.
%\end{prop}
%%%%%%%%%%%%%%%%%%%%%%%%%%%%%%%

A partition $\lambda$ of length $r$ is a non-increasing sequence of $r$ positive integers $(\lambda_1,\dots,\lambda_r)$ with the convention that $\lambda_i=0$ for $i>r$. We identify a partition $\lambda$ with its Young diagram $\{(i,j) \ |\ 1\leq i\leq r, 1\leq j\leq \lambda_i\}$ in English notation. For each vexillary permutation $w \in S_n$, one can assign a partition $\lambda(w)$: let the number of boxes $(i,i+k)$ in the $k$-th diagonal of the Young diagram of $\lambda(w)$ be equal to the number of boxes in the $k$-th diagonal of $D(w)$ for each $k$ (see \cite{KnutsonMillerYong, KnutsonMillerYong2}). This defines a bijection $\phi$ from $D(w)$ to $\lambda(w)$, which takes the $j$-th box in the $k$-th diagonal of $D(w)$ to the $j$-th box in the $k$-th diagonal of $\lambda(w)$. In other words, we set $\phi(p,q) = (p-r_w(p,q), q-r_w(p,q))$ for each $(p,q)\in D(w)$. Under this bijection $\phi$, the boxes in $\Ess(w)$ correspond bijectively to the south-east corners of $\lambda(w)$.
%%%%%%%%%%%%%%%%%%%%%%%%%%%%%%%
\begin{defn}\label{lemVexFilling}
Let $w\in S_n$ be a vexillary permutation.  A subset $\bff(w)=\{(p_i,q_i),i=1,\dots,d\}$ of $\{1,\dots,n\}\times \{1,\dots,n\}$ is called a \emph{flagging set} for $w$ if it satisfies 
\begin{enumerate}[(i)]
\item $p_1\leq p_2 \leq \cdots \leq p_d, \ \ \ q_1\geq q_2 \geq \cdots \geq q_d$;
\item The set $\bff(w)$ contains $\Ess(w)$; 
\item $p_i - r_i = i$ for $i=1,\dots,d$ where $r_i:=r_w(p_i,q_i)$.
\end{enumerate}
\end{defn}
%%%%%%%%%%%%%%%%%%%%%%%%%%%%%%%
Although you can find the following claim in \cite{AndersonFulton, AndersonFulton2}, we give a proof for completeness ({\it cf.} \cite{FlagsFulton, MacdonaldNote}).
\begin{lem}
If $\bff(w)=\{(p_i,q_i),i=1,\dots,d\}$ is a flagging set of a vexillary permutation $w\in S_n$, then the partition $\lambda(w)$ is given by $\lambda_i = q_i-p_i +i$ for each $i=1,\dots, d$. %In particular, $(p_i,q_i)$ and $\phi^{-1}(i,\lambda_i)$ are on the same diagonal for each $i=1,\dots,r$.  
\end{lem}
%%%%%%%%%%%%%%%%%%%%%%%%%%%%%%%
\begin{proof}
If $(p_i,q_i)$ is in $\Ess(w)$, it maps to a southeast corner of $\lambda$ under $\phi$ and hence (iii) implies $\lambda_i = q_i-p_i +i$. Suppose $(p_i,q_i)$ is the first element that is contained in $\Ess(w)$ so that there is no box in the diagram to the east of $(p_i,q_i)$. We see that the condition (iii) implies that $q_j-p_j =\lambda_j-j$ for all $j=1,\dots,i-1$. Next let $(p_i,q_i)$ and $(p_j,q_j)$ be consecutive elements of the flagging set which are contained in $\Ess(w)$ with $i>j$. Consider the $(p_i-p_j+1)\times (q_j-q_i+1)$ rectangle $R$ in $\{1,...,n\}\times \{1,...,n\} $ with $(p_j,q_j)$ as the northeast corner and $(p_i,q_i)$ as the southwest corner. For $i>k>j$, $(p_k,q_k)$ must be in $R$. Since $w$ is vexillary, $\Ess(w)\cap R$ consists of only $(p_i,q_i)$ and $(p_j,q_j)$. This implies that an element of $R$ is in $D(w)$ only if it is on the leftmost column or on the top row. We can also observe that the northwest corner of $R$ is contained in $D(w)$ and connected to $(p_i,q_i)$ or $(p_j,q_j)$. Moreover we see that the number of connected components of $D(w) \cap R$ is at most two. Thus we have one of the following three cases: (1) $D(w)\cap R$ is a hook shape, (2)  $D(w)\cap R$ contains the top row, or (3) $D(w)\cap R$ contains the leftmost column. In each case, we see by the condition (iii) that, for $i>k>j$, $(p_k,q_k)$ and $\phi^{-1}(k,\lambda_k)$ are on the same diagonal, {\it i.e.} $\lambda_k-k=q_k-p_k$. This proves the claim for $r\geq i$ where $r$ is the length of $\lambda(w)$. Suppose that $i>r$. Then $(p_i, q_i)$ must be in the southwest $(n-p_r+1)\times q_r$ rectangle $R'$ containing $(p_r,q_r)$ in $\{1,\dots,n\}\times \{1,\dots, n\}$. Since $q_r$ is the maximum row index of $D(w)$, we can see that $p-r_w(p,q)=i$ if and only if $q-p=-i$. Hence we have $q_i-p_i + i =0$. This completes the proof.
\end{proof}
%%%%%%%%%%%%%%%%%%%%%%%%%%%%%%%
%%%%%%%%%%%%%%%%%%%%%%%%%%%%%%%
\section{Determinant formula of $X_w(E_{\bullet} \to H_{\bullet})$}\label{SecMain}
%%%%%%%%%%%%%%%%%%%%%%%%%%%%%%%%
%%%%%%%%%%%%%%%%%%%%%%%%%%%%%%%%
%%%%%%%%%%%%%%%%%%%%%%%%%%%%%%%%
Let $E_1 \subset \dots \subset E_{n-1} \stackrel{\phi}{\to} H_{n-1} \surj \cdots \surj H_2 \surj H_1$ be a full flag of vector bundles on $X$ followed by a map to a dual full flag. The subscripts indicate the rank. For each $w\in S_{n}$, we define the degeneracy locus $X_w$ in $X$ by
\[
X_w:=X_w(E_{\bullet} \to H_{\bullet}) := \{ x \in X \ |\ \rk(E_q|_x \to H_p|_x) \leq r_w(p,q) \ \forall p,q\}.
\]
By Proposition 4.2 \cite{FlagsFulton}, it suffices to consider the rank conditions for $(p,q) \in \Ess(w)$. In what follows, we assume that the bundles and maps are sufficiently generic so that the codimension of $X_w$ is the length $\ell(w)$ of $w$. 
%%%%%%%%%%%%%%%%%%%%%%%%%%%%%%%%
Suppose that $w$ is vexillary and choose a flagging set $\{(p_i,q_i), i=1,\dots,d\}$ of $w$, then we can write
\[
X_w = \{ x \in X \ |\ \rk(E_{q_i}|_x \to H_{p_i}|_x) \leq r_w(p_i,q_i) \ \forall i=1,\dots,d\}.
\]
%%%%%%%%%%%%%%%%%%%%%%%%%%%%%%%%
The following is our main theorem.
%%%%%%%%%%%%%%%%%%%%%%%%%%%%%%%%
\begin{thm}\label{thmMainA}
Let $w \in S_n$ be a vexillary permutation and $\bff(w)=\{(p_i,q_i); i=1,\dots,d\}$ an arbitrary flagging set of $w$.  Let $\scA_{m}^{[i]}:=\scA_{m}^{[i]}(\bff(w)):=\scS_{m}(H_{p_i}-E_{q_i})$ for each $m\in \ZZ$. Then the class of the degeneracy locus $X_w$ in $\CK^*(X)$ is given by
\begin{equation}\label{maineq1}
[X_w] = \det\left(  \sum_{s=0}^{\infty} \binom{i-d}{s}\beta^s \scA_{\lambda_i+j-i+s}^{[i]} \right)_{1\leq i,j \leq d}
\end{equation}
and also by
\begin{equation}\label{maineq2}
[X_w] = \det\left(  \sum_{s=0}^{\infty} \binom{i-j}{s}\beta^s \scA_{\lambda_i+j-i+s}^{[i]} \right)_{1\leq i,j \leq d},
\end{equation}
where $\lambda(w)=(\lambda_1,\dots,\lambda_d)$.
\end{thm}
%%%%%%%%%%%%%%%%%%%%%%%%%%%%%%%%
\begin{rem}
Independently from our work, Anderson \cite{Anderson2016} obtained a determinant formula of $[X_w]$ in terms of Chern classes, which is equivalent to (\ref{maineq2}) via writing Segre classes in terms of Chern classes. In particular, his formula also gives a formula which depends only on the data of the essential set of $w$.
\end{rem}
%%%%%%%%%%%%%%%%%%%%%%%%%%%%%%%%
\subsection{Resolution of $X_w$}
%%%%%%%%%%%%%%%%%%%%%%%%%%%%%%%%
%%%%%%%%%%%%%%%%%%%%%%%%%%%%%%%%
Consider the dual flag bundle $\pi: \Fl^*(H_{\bfp}) \to X$ whose fiber at $x$ consists of dual flags $D_d|_x \surj \cdots \surj D_1|_x$ with $\dim D_i|_x=i$ and the commutative diagram of surjective maps
\[
\xymatrix{
H_{p_d}|_x \ar[r] \ar[d]& H_{p_{d-1}}|_x \ar[r]  \ar[d]&\cdots \ar[r] &  H_{p_2}|_x \ar[r] \ar[d]& H_{p_1}|_x \ar[d]\\ 
D_{d}|_x \ar[r] & D_{d-1}|_x \ar[r]  &\cdots \ar[r] &  D_2|_x \ar[r] & D_1|_x. 
}
\]
Let $D_i$ be the tautological bundle of rank $i$. It can be constructed as the following tower
\[
\Fl^*(H_{\bfp}):=\PP^*(\ker(H_{p_d} \to D_{d-1}))\to \cdots \to \PP^*(\ker(H_{p_2} \to D_1))\to \PP^*(H_{p_1}) \to X.
\]
We can regard $\ker(D_i \to D_{i-1})$ as the tautological quotient line bundle of $\PP^*(\ker(H_{p_i} \to D_{i-1}))$ where we set $D_0=0$. 
%%%%%%%%%%%%%%%%%%%%%%%%%%%%%%%
\begin{defn}\label{defnY}
For each flagging set $\bff(w)=\{(p_i,q_i), i=1,\dots,d\}$, define a sequence of subvarieties $Y_d \subset \cdots \subset Y_1 \subset \Fl^*(H_{\bfp})$ by
\[
Y_j:= \{ (x, D_{\bullet}|_x)  \in \Fl^*(H_{\bfp})\ |\ \rk(E_{q_i}|_x \to D_i|_x) = 0, \ 1 \leq i \leq j \}.
\]
Denote $Y_d$ by $Y_{\bff(w)}$.
\end{defn}
%%%%%%%%%%%%%%%%%%%%%%%%%%%%%%%
It is well known that $\pi$ maps $Y_{\bff(w)}$ to $X_w$ birationally and that $X_w$ has at worst rational singularities,  therefore $\pi_*[Y_{\bff(w)}]=[X_w]$ in $\CK^*(X)$.
%%%%%%%%%%%%%%%%%%%%%%%%%%%%%%%
\begin{lem}\label{lemYd}
In $\CK^*(\Fl^*(H_{\bfp}))$, we have
\[
[Y_{\bff(w)}] = \prod_{i=1}^r c_{q_i}(\ker(D_i \to D_{i-1}) \otimes E_{q_i}^{\vee}).
\]
\end{lem}
%%%%%%%%%%%%%%%%%%%%%%%%%%%%%%%
\begin{proof}
By definition, there is an obvious bundle map $E_{q_i} \to \ker(D_j\to D_{j-1})$ over $Y_{j-1}$ and its zero locus in $Y_{j-1}$ coincides with $Y_j$. Thus from the standard fact ({\it cf.} \cite[Lemma 2.2]{HIMN}) we obtain 
\[
[Y_j] = c_{q_j}(\ker(D_j\to D_{j-1})\otimes E_{q_j}^{\vee})
\]
in $\CK^*(Y_{j-1})$. The claim follows from the projection formula.
\end{proof}
%%%%%%%%%%%%%%%%%%%%%%%%%%%%%%%%
%For a surjective map $E\surj F$, $\ker(E\surj F) = E - F$ in the Grothendieck ring of $X$. 
%%%%%%%%%%%%%%%%%%%%%%%%%%%%%%%%
%%%%%%%%%%%%%%%%%%%%%%%%%%%%%%%%
\subsection{Pushforward formula}
%%%%%%%%%%%%%%%%%%%%%%%%%%%%%%%
%%%%%%%%%%%%%%%%%%%%%%%%%%%%%%%
Set $R=\CK^*(X)$. This is a graded algebra over $\ZZ[\beta]$. Let $t_1,\ldots,t_d$ be indeterminates of degree $1$. We use the multi-index notation $t^\bfs:=t_1^{s_1}\cdots t_d^{s_d}$ for $\bfs=(s_1,\dots,s_d)\in \ZZ^d$. A formal Laurent series 
\[
f(t_1,\ldots,t_d)=\sum_{\bfs\in\ZZ^d}a_{\bfs}t^{\bfs}
\]
is {\em homogeneous of degree} $m\in \ZZ$ if $a_{\bfs}$ is zero unless $a_{\bfs}\in R_{m-|\bfs|}$ with $|\bfs|=\sum_{i=1}^d s_i$. Let $\supp f = \{\bfs \in \ZZ^d \ |\ a_{\bfs}\not=0\}$.
%%%%%%%%%%%%%%%%%%%%%%%%%%%%%%
For each $m \in \ZZ$, define $\calL^{R}_m$ to be the space of all formal Laurent series of homogeneous degree $m$ such that there exists $\sfn\in \ZZ^d$ such that $\sfn + \supp f$ is contained in the cone in $\ZZ^d$ defined by $s_1\geq0, \; s_1+s_2\geq 0, \;\cdots, \; s_1+\cdots + s_d \geq 0$. Then $\calL^{R}:=\bigoplus_{m\in \ZZ} \calL^{R}_m$ is a graded ring over $R$ with the obvious product. 
%%%%%%%%%%%%%%%%%%%%%%%%%%%%%%%
For each $i=1,\dots, d$, let $\calL^{R,i}$ be the $R$-subring of $\calL^R$ consisting of series that do not contain any negative powers of $ t_1,\dots, t_{i-1}$.  In particular, $\calL^{R,1}=\calL^{R}$. 
%%%%%%%%%%%%%%%%%%%%%%%%%%%%%%%
A series $f(t_1,\ldots,t_d)$ is a {\em power series} if it doesn't contain any negative powers of $t_1,\dots,t_d$. Let $R[[t_1,\ldots,t_d]]_{m}$ denote the set of all power series in $t_1,\dots, t_d$ of degree $m\in \ZZ$. We set $R[[t_1,\ldots,t_d]]_{\gr}:=\bigoplus_{m\in \ZZ}R[[t_1,\ldots,t_d]]_{m}$.
%%%%%%%%%%%%%%%%%%%%%%%%%%%%%%%
\begin{defn}
For each $j=1,\dots, d$, define a graded $R$-module homomorphism
\[
\phi_j: \calL^{R,j} \to \CK^*(\PP^*(\ker(H_{p_{j-1}} \to D_{j-2}))
\]
by setting $\phi_j( t_1^{s_1}\cdots  t_d^{s_d})= \tau_1^{s_1}\cdots \tau_{j-1}^{s_{j-1}}\scA_{s_j}^{[j]}\cdots \scA_{s_d}^{[d]}$ where $\tau_i:=c_1(\ker(D_i \to D_{i-1}))$.
\end{defn}
%%%%%%%%%%%%%%%%%%%%%%%%%%%%%%%
\begin{lem}\label{lempush}
Let $\pi_j: \PP^*(\ker(H_{p_j} \to D_{j-1}))\to \PP^*(\ker(H_{p_{j-1}} \to D_{j-2}))$. We have
\[
\pi_{j*}(\tau^sc_{q_j}(\ker(D_j \to D_{j-1}) \otimes E_{q_j}^{\vee})) =  \phi_j\left(t_j^{\lambda_j} \prod_{1\leq i< j} \frac{1 - t_i/t_j}{1+\beta t_i} \right),
\]
in $\CK^*(\PP^*(\ker(H_{p_{j-1}} \to D_{j-2}))$.
\end{lem}
%%%%%%%%%%%%%%%%%%%%%%%%%%%%%%%
\begin{proof}
%Recall (\ref{eqrelS1})
%\[
%\scS_m(E-F) =\frac{1}{c(F;\beta)}\sum_{p=0}^f (-1)^p c_p(F)\scS_{m-p}(E).
%\]
By applying Proposition \ref{push of tensor}, we have
\begin{eqnarray*}
\pi_{j*}(\tau^sc_{q_j}(\ker(D_j \to D_{j-1})\otimes E_{q_j}^{\vee})) 
%&=& \scS_{q_j-p_j+j}(H_{p_j} - D_{j-1} - E_{q_j})\\
&=&\scS_{\lambda_j}(H_{p_j}- E_{q_j} -D_{j-1})\\
&=& \frac{1}{c(D_{j-1};\beta)}\sum_{p=0}^{j-1} (-1)^p c_p(D_{j-1})\scS_{\lambda_j-p}(H_{p_j}- E_{q_j}),
\end{eqnarray*}
where the second equality follows from (\ref{eqrelS1}). Thus by using $\phi_j$, we obtain the desired formula where we have used the identity
\begin{eqnarray*}
\frac{1}{\prod_{1\leq i\leq j-1} (1+\beta t_i)} \sum_{p=0}^{j-1} (-1)^p e_p(t_1,\dots,t_{j-1}) t_j^{\lambda_j-p}
=t_j^{\lambda_j} \prod_{1\leq i< j} \frac{1 - t_i/t_j}{1+\beta t_i}.
\end{eqnarray*}
\end{proof}
%%%%%%%%%%%%%%%%%%%%%%%%%%%%%%%
By applying Lemma \ref{lempush} to Lemma \ref{lemYd} successively, we obtain the following proposition.
\begin{prop}\label{propMainA}
We have
\[
\pi_*[Y_{\bff(w)}] = \phi_1\left(\prod_{j=1}^d t_j^{\lambda_j} \cdot \prod_{1\leq i< j\leq d} \frac{1 - t_i/t_j}{1+\beta t_i} \right).
\]
\end{prop}
%%%%%%%%%%%%%%%%%%%%%%%%%%%%%%%
%%%%%%%%%%%%%%%%%%%%%%%%%%%%%%%
\subsection{Proof of Theorem \ref{thmMainA}}
%%%%%%%%%%%%%%%%%%%%%%%%%%%%%%%
%%%%%%%%%%%%%%%%%%%%%%%%%%%%%%%
By the formula of the Vadermonde determinant, we have
\[
\prod_{j=1}^d t_j^{\lambda_j} \cdot \prod_{1\leq i< j\leq d} \frac{1 - t_i/t_j}{1+\beta t_i}  = \det\left( t_i^{\lambda_i+j-i} \left(1+\beta t_i\right)^{i-r} \right).
\]
Since $\pi_*[Y_{\bff(w)}]=[X_w]$, Proposition \ref{propMainA} implies that
\[
[X_w] = \det\left(\phi_1\left(t_i^{\lambda_i+j-i} \left(1+\beta t_i\right)^{i-d}\right)\right).
\]
By writing
\[
t_i^{\lambda_i+j-i} \left(1+\beta t_i\right)^{i-d}  = \sum_{s=0}^{\infty} \binom{i-d}{s} \beta^s t_i^{\lambda_i+j-i+s},
\]
and applying $\phi_1$, we obtain (\ref{maineq1}). 

The proof of (\ref{maineq2}) is analogous. We replace Lemma \ref{lempush} with  
\[
\pi_{j*}(\tau^sc_{q_j}(\ker(D_j \to D_{j-1}) \otimes E_{q_j}^{\vee})) =  \phi_j\left(t_j^{\lambda_j} \prod_{1\leq i< j} (1 - \bar t_i/\bar t_j) \right), 
\]
which can be proved similarly but one uses (\ref{eqrelS2}) instead of (\ref{eqrelS1}). Then, instead of Proposition \ref{propMainA}, we have 
\[
\pi_*[Y_{\bff(w)}] = \phi_1\left(\prod_{j=1}^d t_j^{\lambda_j} \cdot \prod_{1\leq i< j\leq d} (1 - \bar t_i/\bar t_j) \right).
\]
Now by the formula of the Vandermonde determinant, we obtain (\ref{maineq2}) ({\it cf.} \cite[Theorem 3.10.]{HIMN}).
%%%%%%%%%%%%%%%%%%%%%%%%%%%%%%%
%%%%%%%%%%%%%%%%%%%%%%%%%%%%%%%
\section{Vexillary double Grothendieck polynomials}\label{SecGro}
%%%%%%%%%%%%%%%%%%%%%%%%%%%%%%%
%%%%%%%%%%%%%%%%%%%%%%%%%%%%%%%
%%%%%%%%%%%%%%%%%%%%%%%%%%%%%%%%
For a given partition $\lambda$, a flagging $f$ of $\lambda$ is a sequence of natural numbers $(f_1,\dots,f_r)$ where $r$ is the length of $\lambda$. A flagged set-valued tableaux in shape $\lambda$ with flagging $f$ is a set-valued tableaux of shape $\lambda$ such that the numbers used in the $i$-th row are at most $f_i$. Let $\bff(w)=\{(p_i,q_i), i=1,\dots, r\}$ be a flagging set of $w$ and consider the flagging $f_w$ with $(f_w)_i:=p_i$. Let $\FSVT(w)$ be the set of all flagged set-valued tableaux of $w$ with the flagging $f_w$. It is worth remaking that the set $\FSVT(w)$ is the same for any choice of flagging $f$ such that $p'\leq f_i\leq p$ where $p'$ is the row index of $\phi^{-1}(i,\lambda_i)$ and $p$ is the row index of the preimage under $\phi$ of the south-east corner of $\lambda$ whose column contains $(i,\lambda_i)$. 

Let $x=(x_1,x_2,\dots )$ and $b=(b_1,b_2,\dots)$ be sets of infinitely many indeterminants. In \cite{KnutsonMillerYong} (see also \cite{KnutsonMillerYong2}), Knutson--Miller--Yong described the double Grothendick polynomials of Lascoux--Sch{\"u}tzenberger \cite{LascouxSchutzenberger3} associated to a vexillary permutation $w$ as a generating function of flagged set-valued tableaux. Namely we have
\[
G_w(x|b) = \sum_{\tau \in \FSVT(w)} (-1)^{|\tau| - |\lambda(w)|} \prod_{e\in \tau} (x_{val(e)} \oplus y_{val(e) + c(e)-r(e)}),
\]
where $val(e)$ is the numerical value of $e$ and $c(e)$ and $r(e)$ are the column and row indices of $e$ respectively. Note that we replaced $1-\mathbf{x_i}$ and $1-\mathbf{y_i}^{-1}$ in \cite{KnutsonMillerYong} by $x_i$ and  by $b_i$ respectively. We know from the work of Fulton--Lascoux \cite{FultonLascoux} and Buch \cite{BuchQuiver} that 
\begin{equation}\label{XG}
[X_w] = G_w(x|b),
\end{equation}
where we set $x_i = c_1(\ker(H_i\to H_{i-1}))$ and $b_i=c_1((E_i/E_{i-1})^{\vee})$. 
%%%%%%%%%%%%%%%%%%%%%%%%%%%%%%%%
%%%%%%%%%%%%%%%%%%%%%%%%%%%%%%%%
%%%%%%%%%%%%%%%%%%%%%%%%%%%%%%%%
%%%%%%%%%%%%%%%%%%%%%%%%%%%%%%%%

For nonnegative integers $p,q$ and an integer $m$, we define $G_m^{[p,q]}(x|b)$ by
\[
\sum_{m\in \ZZ} G_m^{[{p},{q}]}(x|b) u^m = \frac{1}{1+\beta u^{-1}} \prod_{1\leq i \leq {p}}\frac{1+\beta x_i}{1-x_iu} \prod_{1\leq j \leq {q}} (1+b_j(\beta + u)).
\]
We see from Definition \ref{defnSeg} that $G_m^{[{p},{q}]}(x|b)=\scS_m(E-F)$ for vector bundles $E$ and $F$ of rank $p$ and $q$ with Chern roots $\{x_i\}$ and $\{\bar b_i\}$ respectively. Consequently we can derive from Theorem \ref{thmMainA} the following Jacobi-Trudi type formulas of vexillary double Grothendieck polynomials. 
%%%%%%%%%%%%%%%%%%%%%%%%%%%%%%%%
\begin{thm}
Let $w \in S_n$ be a vexillary permutation. For an arbitrary flagging set $\{(p_i,q_i), i=1,\dots,d\}$ of $w$,  we have
\[
G_{w}(x|b) = \det\left(  \sum_{s=0}^{\infty} \binom{i-d}{s}\beta^s G_{\lambda_i+j-i+s}^{[p_i,q_i]}(x|b) \right)_{1\leq i,j \leq d},
\]
and also
\[
G_{w}(x|b) = \det\left(  \sum_{s=0}^{\infty} \binom{i-j}{s}\beta^s G_{\lambda_i+j-i+s}^{[p_i,q_i]}(x|b) \right)_{1\leq i,j \leq d},
\]
where $\lambda(w)=(\lambda_1,\dots,\lambda_d)$.
\end{thm}
%%%%%%%%%%%%%%%%%%%%%%%%%%%%%%%%
%%%%%%%%%%%%%%%%%%%%%%%%%%%%%%%%
%\section{$T$-equivariant connective $K$-theory of $\Fl(\CC^n)$}
%%%%%%%%%%%%%%%%%%%%%%%%%%%%%%%%%
%%%%%%%%%%%%%%%%%%%%%%%%%%%%%%%%%
%Let $\Fl(\CC^n)$ be the full flag variety with the standard action of $T_n=(\CC^{\times})^n$. Let 
%\[
%V_1 \subset V_2 \subset \cdots  \subset V_n
%\]
%be its tautological flag. Let $F^i=\Span(\bfe_{i+1},\dots,\bfe_n)$ and $F_i:=\CC^n/F^i$ where $\{\bfe_i\}$ is the standard basis of $\CC^n$. We denote also by $F_i$ the bundle over $\Fl(\CC^n)$ with fiber $F_i$. Consider the sequence $F_{\bullet}^{\vee} \to V_{\bullet}^{\vee}$
%\[
%F_1^{\vee} \subset \cdots \subset F_n^{\vee} = V_n^{\vee} \surj \cdots \surj V_1^{\vee}.
%\]
%We can define the Schubert variety associated $X_w$ to $w\in S_n$ in $\Fl(\CC^n)$ by $X_w(F_{\bullet}^{\vee} \to V_{\bullet}^{\vee})$. Set
%\begin{eqnarray*}
%x_i&:=&c_1(\ker(V_{i}^{\vee} \to V_{i-1}^{\vee}))=c_1((V_i/V_{i-1})^{\vee})\\
%b_i&:=&c_1((F_i^{\vee}/F_{i-1}^{\vee})^{\vee}) = c_1(\ker(F_i \to F_{i-1}))=c_1(\ker(E/F^i \to E/F^{i-1})) = c_1(F^{i-1}/F^i).
%\end{eqnarray*}
%We have $[X_w] = G_w(x|b)$.
%%%%%%%%%%%%%%%%%%%%%%%%%%%%%%%%
%%%%%%%%%%%%%%%%%%%%%%%%%%%%%%%%
\section{Generalization to algebraic cobordism}\label{SecAlgCob}
%%%%%%%%%%%%%%%%%%%%%%%%%%%%%%%%
%%%%%%%%%%%%%%%%%%%%%%%%%%%%%%%%
In this section, we discuss the generalization of Theorem \ref{thmMainA} to the algebraic cobordism of Levine--Morel \cite{LevineMorel}. Let $\Omega^*(X)$ denote the algebraic cobordism of $X$.

First we recall, from \cite{HudsonMatsumura}, the formulas for the Segre classes in algebraic cobordism. Let $F_\Omega(u,v)$ be the formal group law of algebraic cobordism, the universal one defined over the Lazard ring $\LL$. We denote the formal inverse of $x$ by $\chi(x)$. Let $P(z,x)$ be the unique power series in $x$ and $z$ defined by $F_{\Omega}(z,\chi(x))=(z-x)P(z,x)$. Let $E$ and $F$ denote vector bundles on $X$ of rank $e$ and $f$ respectively. For each integer $s\geq 0$, we define the class $w_{-s}(E)$ in $\Omega^{-s}(X)$ by the following generating function
\[
w(E;u):=\sum_{s= 0}^{\infty} {w}_{-s}(E) u^{-s} := \prod_{q=1}^e P(u^{-1}, x_{q}),
\]
where $x_q$ for $q=1,\dots,e$ are Chern roots of $E$. Moreoever, let 
\[
\scP(u):=\sum_{i=0}^{\infty} [\PP^i]u^{-i},
\]
where $[\PP^i]$ is the class of the projective space $\PP^i$ of degree $-i$ in $\LL$. We can define the relative Segre classes of a virtual bundle $E-F$ in $\Omega^*(X)$ by the following generating function:
\begin{equation}\label{relSegALG}
\scS(E-F;u) := \sum_{k\in \ZZ} \scS_k(E-F)u^k:= \frac{\scP(u)}{c(E-F;-u)w(E-F;u)}.%:= \scS(V;u)c(W;-u)w(W;u).
\end{equation}
%We have
%\begin{equation}\label{relSegALG2}
%\scS_k(E-F) := \sum_{q=0}^{\rk(F)}\sum_{j = 0}^{\infty} (-1)^qc_q(F)w_{-j}(F)\scS_{k-q+j}(E).
%\end{equation}
%%%%%%%%%%%%%%%%%%%%%%%%%%%%%%%%
By Theorem 3.9 in \cite{HudsonMatsumura}, we can express the relative Segre classes by the pushforward of Chern classes along a projective bundle as follows. Let $\pi: \PP^*(E) \to X$ be the dual projective bundle of $E$, $\calQ$ its tautological quotient line bundle, and $\tau:=c_1(\calQ)$. We have
\begin{equation}\label{pushSegAlg}
\pi_*(\tau^s c_f(\calQ \otimes F^{\vee}))  = \scS_{f-e+1+s}(E-F).
\end{equation}
We see from this that the Segre class $\scS_m(E)$ for a vector bundle $E$ is the natural generalization of the one given by Fulton in \cite{FultonIntersection}.
%%%%%%%%%%%%%%%%%%%%%%%%%%%%%%%%

Now we describe the main theorem in this section. Let $\bff(w)=\{(p_i,q_i), i=1,\dots, d\}$ be a flagging set of a vexillary permutation $w\in S_n$. We define $Y_{\bff(w)}$ as in Definition \ref{defnY} and also let $\scA_{m}^{[i]}:=\scA_{m}^{[i]}(\bff(w)):=\scS_{m}(H_{p_i}-E_{q_i})$ in $\Omega^*(X)$ as before. If we set 
\[
\prod_{1\leq i<j\leq d}P(t_j,t_i)=\sum_{\bfs=(s_1,\dots, s_d) \in\ZZ_{\geq 0}^r}a_{\bfs}\cdot t_1^{s_1}\cdots t_d^{s_d}, \ \ \ a_{\bfs} \in \LL
\]
and denote the Schur determinant by 
\[
\Delta_{\bfm}(\scA^{[1]}, \dots, \scA^{[d]}) := \det\left(\scA_{m_i+j-i}^{[i]}\right)_{1\leq i,j\leq d}
\]
for each $\bfm=(m_1,\dots,m_d) \in \ZZ^d_{\geq 0}$, we can describe the class of $Y_{\bff(w)}$ in $\Omega^*(X)$ as follows. Let us stress that the class $[Y_{\bff(w)} \to X]$ in $\Omega^*(X)$ does depend on the different choice of a flagging set $\bff(w)$ of $w$. %For example, we can take a non-trivial flagging set $\bff(\id)$ for an identity element $\id$ of $S_n$ and we can see that the class $[Y_{\bff(w)} \to X]$ is not $1$. 
%%%%%%%%%%%%%%%%%%%%%%%%%%%%%%%%
\begin{thm}\label{thmMainAALG}
For each flagging set $\bff(w)=\{(p_i,q_i) \ |\ i=1,\dots, d\}$ of a vexillary permutation $w\in S_n$, we have
\[
[Y_{\bff(w)} \to X] = \sum_{\bfs=(s_1,\dots, s_d) \in\ZZ^d_{\geq 0}}a_{\bfs} \Delta_{\lambda(w)+\bfs}(\scA^{[1]}, \dots, \scA^{[d]}). 
\]
\end{thm}
%%%%%%%%%%%%%%%%%%%%%%%%%%%%%%%%
\begin{proof}
The proof is similar to the one for connective $K$-theory in Section \ref{SecMain} and to the one in \cite{HudsonMatsumura}. Here we only give an outline of the proof and the details are left to the readers. In $\Omega^*(\Fl^*(H_{\bfp}))$, we have
\begin{equation}\label{alg1eq}
[Y_{\bff(w)} \to \Fl^*(H_{\bfp})] = \prod_{i=1}^d c_{q_i}(\ker(D_i \to D_{i-1}) \otimes E_{q_i}^{\vee}).
\end{equation}
As before, the graded $R$-module homomorphism $\phi_j: \calL^{R,j} \to \Omega^*(\PP^*(\ker(H_{p_{j-1}} \to D_{j-2}))$ is defined by setting $\phi_j( t_1^{s_1}\cdots  t_d^{s_d})= \tau_1^{s_1}\cdots \tau_{j-1}^{s_{j-1}}\scA_{s_j}^{[j]}\cdots \scA_{s_d}^{[d]}$. Then, similarly to Lemma \ref{lempush} and \cite[Lemma 4.6]{HudsonMatsumura}, the definition of $\phi_j$ together with (\ref{pushSegAlg}) allows us to obtain
\begin{equation}\label{alg2eq}
\pi_{j*}(\tau^sc_{q_j}(\ker(D_j \to D_{j-1}) \otimes E_{q_j}^{\vee})) =  \phi_j\left(t_j^{\lambda_j} \prod_{1\leq i< j} (1 - t_i/t_j)P(t_j, t_i) \right).
\end{equation}
By applying $\pi_*$ to (\ref{alg1eq}) and using (\ref{alg2eq}) successively, we have
\[
[Y_{\bff(w)} \to X]=\pi_*[Y_{\bff(w)} \to \Fl^*(H_{\bfp})] = \phi_1\left(\prod_{j=1}^d t_j^{\lambda_j} \cdot \prod_{1\leq i< j\leq d} (1 - t_i/t_j)P(t_j,t_i)\right).
\]
By Vandermonde's determinant formula and by applying $\phi_1$, we obtain the desired formula.
\end{proof}

\vspace{3mm}

\noindent\textbf{Acknowledgements.} 
We would like to thank David Anderson for his helpful comments on the earlier versions of the manuscript. The second author is supported by Grant-in-Aid for Young Scientists (B) 16K17584.

%A considerable part of this work developed while the first and third authors were affiliated to KAIST, which they would like to thank for the excellent working conditions.
%Part of this work developed while the first author was affiliated to  POSTECH, which he would like to thank for the excellent working conditions. He would also like to gratefully acknowledge the support of the National Research Foundation of Korea (NRF) through the grants funded by the Korea government (MSIP) (ASARC, NRF-2007-0056093) and (MSIP)(No.2011-0030044).
%The third author is supported by 	Grant-in-Aid for Young Scientists (B) 16K17584.
%%%%%%%%%%%%%%%%%%%%%%%%%%%%%%%%
%%%%%%%%%%%%%%%%%%%%%%%%%%%%%%%%
%%%%%%%%%%%%%%%%%%%%%%%%%%%%%%%%
%%%%%%%%%%%%%%%%%%%%%%%%%%%%%%%%
%%%%%%%%%%%%%%%%%%%%%%%%%%%%%%%%
%%%%%%%%%%%%%%%%%%%%%%%%%%%%%%%%
%%%%%%%%%%%%%%%%%%%%%%%%%%%%%%%%
%%%%%%%%%%%%%%%%%%%%%%%%%%%%%%%%
%%%%%%%%%%%%%%%%%%%%%%%%%%%%%%%%
%%%%%%%%%%%%%%%%%%%%%%%%%%%%%%%%
%%%%%%%%%%%%%%%%%%%%%%%%%%%%%%%%
%%%%%%%%%%%%%%%%%%%%%%%%%%%%%%%%
%%%%%%%%%%%%%%%%%%%%%%%%%%%%%%%%
%%%%%%%%%%%%%%%%%%%%%%%%%%%%%%%%
%%%%%%%%%%%%%%%%%%%%%%%%%%%%%%%%
%%%%%%%%%%%%%%%%%%%%%%%%%%%%%%%%
%%%%%%%%%%%%%%%%%%%%%%%%%%%%%%%%
%%%%%%%%%%%%%%%%%%%%%%%%%%%%%%%%
%%%%%%%%%%%%%%%%%%%%%%%%%%%%%%%%
%   		     BIBLIOGRAPHY 
%%%%%%%%%%%%%%%%%%%%%%%%%%%%%%%
\bibliography{references}{}
\bibliographystyle{acm}
%%%%%%%%%%%%%%%%%%%%%%%%%%%%%%%
%%%%%%%%%%%%%%%%%%%%%%%%%%%%%%%
%%%%%%%%%%%%%%%%%%%%%%%%%%%%%%%
\begin{small}
{\scshape
\noindent Thomas Hudson, Fachgruppe Mathematik
und Informatik, Bergische Universit\"{a}t Wuppertal, Gaustrasse 20, 42119 Wuppertal, Germany
}
\end{small}

{\textit{email address}: \tt{hudson@math.uni-wuppertal.de}}
 
\begin{small}
{\scshape
\noindent Tomoo Matsumura, Department of Applied Mathematics, Okayama University of Science, Okayama 700-0005, Japan
}
\end{small}

{\textit{email address}: \tt{matsumur@xmath.ous.ac.jp}}

%%%%%%%%%%%%%%%%%%%%%%%%%%%%%%%
%%%%%%%%%%%%%%%%%%%%%%%%%%%%%%%
%%%%%%%%%%%%%%%%%%%%%%%%%%%%%%%
\end{document}